\documentclass[12pt,a4paper,reqno]{amsart}
\usepackage{graphics,epic}
\usepackage{amsmath,amssymb, amsthm}
\usepackage[all,2cell]{xy}
\usepackage{color}

\newtheorem{thm}{Theorem}[section]
\newtheorem*{theorem*}{Theorem}

\newtheorem{lem}[thm]{Lemma}
\newtheorem{prop}[thm]{Proposition}
\newtheorem{coro}[thm]{Corollary}

\newtheorem*{conjecture*}{Conjecture}

\newcommand\mtx[1]{\begin{bmatrix} #1 \end{bmatrix}}
\newcommand{\pmtx}[1]{\begin{pmatrix}#1\end{pmatrix}}
\newcommand\Itw[3]{\pmtx{ #1 \\ #2\,\, #3}}

\newcommand{\cc}{{\mathcal C}}

\newcommand{\ch}{{\mathcal H}}

\newcommand{\E}{{\mathbb E}}

\newcommand{\fusion}[3]{{\binom{#3}{#1\;#2}}}

\renewcommand{\hat}[1]{\widehat{#1}}

\newcommand{\Hom}{{\rm Hom}\,}

\newcommand{\FPdim}{{\rm FPdim}\,}

\newcommand{\Z}{\mathbb{Z}}

\newcommand{\C}{\mathbb{C}}

\def\C{{\mathbb C}}

\def\R{{\mathbb R}}

\def\Z{{\mathbb Z}}

\def\H{{\mathbb H}}

\def\1{{\bf 1}}
\def\tr{{\rm tr}}

\def \Hom{{\rm Hom}}

\def \a{\alpha}

\def \h{\mathfrak{h}}
\def \l{\lambda}

\def \g{\mathfrak{g}}
\def\wh{\widehat{\mathfrak h}}
\def\wg{\widehat{\mathfrak g}}

\def\qdim{{\rm qdim}}
\def\<{\langle}
\def\>{\rangle}
\def\cc{\mathcal C}
\def\ch{{\rm ch}}
\def\E{{\mathbb E}}

\newcommand\CC{\mathcal C}

\begin{document}
\title[Parafermion vertex operator algebras]{The irreducible modules and fusion rules for the parafermion vertex operator algebras}
\author{Chunrui Ai}
\address{School of Mathematics, Sichuan University,
 Chengdu 610064 China}
 \email{aichunrui0908@126.com}
\author{Chongying Dong}
\address{Department of Mathematics, University of California, Santa Cruz, CA 95064}
 \email{dong@ucsc.edu}
\thanks{The second author was supported by NSF grant DMS-1404741 and China NSF grant 11371261}
\author{Xiangxu Jiao}
\address{Department of Mathematics, East China Normal University, Shanghai 200241, China}
\email{xyjiao@math.ecnu.edu.cn}
\thanks{The third author was supported by  China NSF grant 11401213}
\author{Li Ren}
\address{School of Mathematics, Sichuan University,
 Chengdu 610064 China }
 \email{renl@scu.edu.cn}
 \thanks{The fourth author was supported by China Postdoctor special funding 2014T70866 and China NSF grant 11301356}
\maketitle

\begin{abstract}
The irreducible modules for the parafermion vertex operator algebra associated to any finite dimensional Lie algebra $\g$ and any positive integer $k$ are classified,  the quantum dimensions are computed and the fusion rules are determined.
\end{abstract}

\section{Introduction}

This paper is a continuation of \cite{DR} and \cite{DW4} on the parafermion vertex operator algebra $K(\g,k)$ associated to any finite dimensional simple Lie algebra $\g$ and positive integer $k.$ In particular, we identify the irreducible modules listed in \cite{DR}, compute the quantum dimensions and determine the fusion rules for $K(\g,k).$
In the case $\g=sl_2,$ these results have been obtained previously in \cite{ALY1} and \cite{DW4}.

Closely related to the $Z$-algebras \cite{LP}, \cite{LW1}, \cite{LW2} , the parafermion conformal field theory \cite{ZF} and the generalized vertex operator algebras
\cite{DL}, the parafermion vertex operator algebra $K(\g,k)$ is the commutant of the Heisenberg vertex operator algebra associated to the Cartan subalgebra $\h$ of $\g$ in the affine vertex operator algebra $L_{\wg}(k,0).$
The structure of the parafermion vertex operator algebras has been studied extensively in \cite{DLY}, \cite{DLWY},
\cite{DW1}. The representation theory of $K(\g,k)$ has also been understood well due to the recent work \cite{DW2},
\cite{ALY1}-\cite{ALY2} and \cite{DR}. Precisely, $K(\g,k)$ is $C_2$-cofinite \cite{ALY1},  \cite{M}, the irreducible modules
for $K(sl_2,k)$ are classsified   and rationality of $K(sl_2,k)$  is obtained in \cite{ALY1}-\cite{ALY2}.
The  irreducible modules and the rationality of $K(\g,k)$ for general $\g$ are determined  in \cite{DR} with the help from \cite{M}, \cite{KM} and \cite{CM}.

The quantum dimensions of modules for vertex operator algebras \cite{DJX1} or Frobenius-Perron dimensions for fusion category  \cite{ENO} play essential roles in this paper. According to \cite{H2}, \cite{H3}, the category $\CC_V$ of modules of a rational and $C_2$-cofinite vertex operator algebra $V$ under the tensor product defined in \cite{HL1, HL2, HL3, H1} is a braided fusion tensor category over $\C$. It is essentially proved in \cite{DJX1} that the quantum dimension
of irreducible module for $V$ is exactly the Frobenius-Perron dimension of the simple object in the category $\CC_V.$
This enables us to freely use the quantum dimensions and Frobenius-Perron dimension whenever it is convenient.

The quantum dimensions of irreducible $K(\g,k)$-modules are computed first. It is well known that
the irreducible modules for rational vertex operator algebra $L_{\wg}(k,0)$ are exactly the level $k$ integral highest weight modules $L_{\wg}(k,\Lambda)$ where $\Lambda$ is a dominant weight of $\g$ such that
$(\Lambda,\theta)\leq k$ and $\theta$ is the maximal root of $\g$ and $(\theta,\theta)=2$ \cite{FZ}, \cite{LL}.
The set of such $\Lambda$ is denoted by $P_+^k.$  Let $Q$ be the root lattice of $\g$ and $Q_L$ be the sublattice of $Q$ spanned by the long roots of $\g.$ As we will see, the dual lattice $Q_L^{\circ}$ is exactly
the weight lattice $P$ of $\g.$ The lattice vertex operator algebra $V_{\sqrt{k}Q_L}$ and $V_{\sqrt{k}Q_L}\otimes K(\g,k)$
are subalgebras of $L_{\wg}(k,0)$ \cite{K}, \cite{DW3}. Then as $V_{\sqrt{k}Q_L}\otimes K(\g,k)$-module, $L_{\wg}(k,\Lambda)$ has decomposition
$$L_{\widehat{\g}}(k,\Lambda)=\bigoplus_{i\in Q/kQ_L}V_{\sqrt{k}Q_L+\frac{1}{\sqrt{k}}(\Lambda+\beta_i)}\otimes M^{\Lambda,\Lambda+\beta_i}$$
where $Q=\bigcup_{i\in Q/kQ_L}(kQ_L+\beta_i)$ and $M^{\Lambda,\Lambda+\beta_i}$ is an irreducible $K(\g,k)$-module \cite{DR}. It turns out that the quantum dimension of each $M^{\Lambda,\Lambda+\beta_i}$ as $K(\g,k)$-module
equals to the quantum dimension of $L_{\wg}(k,\Lambda)$ as $L_{\wg}(k,0)$-module. It follows from \cite{C}
that the quantum dimension of $M^{\Lambda,\Lambda+\beta_i}$ is equal to $$\prod_{\alpha>0}\frac{(\Lambda+\rho,\alpha)_q}{(\rho,\alpha)_q}$$
 where $\alpha>0$ means that $\alpha$ is a positive root and $n_q=\frac{q^n-q^{-n}}{q-q^{-1}}$
 with $q=e^{\frac{\pi i}{k+h^{\vee}}}$ and $h^{\vee}$ is the dual Coxeter number of $\g.$
 This result is very useful in identifying these irreducible $K(\g,k)$-modules.

Let $\theta=\sum_{i=1}^la_i\alpha_i$ where $\{\alpha_1,\cdots,\alpha_l\}$ is the set of simple roots and
let $\Lambda_1,...,\Lambda_l$ be the fundamental weights of $\g.$
According to \cite{L1},\cite{L3}, if $a_i=1$ then $L_{\wg}(k,k\Lambda_i)$ is a simple current and
$L_{\wg}(k,k\Lambda_i)\boxtimes L_{\wg}(k,\Lambda)=L_{\wg}(k,\Lambda^{(i)})$ for any $\Lambda\in P_+^k$
where $\Lambda^{(i)}\in P_+^k$ is uniquely determined by $\Lambda$ and $i.$
One can show that $L_{\wg}(k,\Lambda)$ and $L_{\wg}(k,\Lambda^{(i)})$ are isomorphic $K(\g,k)$-modules.
This gives a nontrivial identification between irreducible $K(\g,k)$-modules $M^{\Lambda,\lambda}$
and $M^{\Lambda^{(i)},\lambda+k\Lambda_i}$ for any $\lambda\in\Lambda+Q.$ As a result,
the set
$$\{M^{\Lambda,\Lambda+\beta_j}|\Lambda\in P_+^k, j\in Q/kQ_L\}$$
has at most $\frac{|P_+^k||Q/kQ_L|}{|P/Q|}$ inequivalent irreducible $K(\g,k)$-modules. Using the relation between
the global dimensions of $L_{\wg}(k,0)$ and $V_{\sqrt{k}Q_L}\otimes K(\g,k)$ from the category theory one can conclude that the identification given in \cite{DR} is complete and $K(\g,k)$ has exactly $\frac{|P_+^k||Q/kQ_L|}{|P/Q|}$ inequivalent irreducible $K(\g,k)$-modules.

For the determination of the fusion rules, the quantum dimensions are used again. The connection between the quantum dimensions and the fusion product is the following equality
$$\qdim_V(M\boxtimes N)=\qdim_VM \cdot \qdim_VN$$
for any rational and $C_2$-cofinite vertex operator algebra $V$ and its irreducible modules $M,N$ \cite{DJX1}.
This quantum dimension equality gives an upper bound for the fusion rules among three irreducible
$V$-modules.  In a normal situation, one has found some intertwining operators among irreducible modules already. Using the upper bounds from the quantum dimension equality, one can check
if these intertwining operators give the enough fusion rules. This is exactly how the fusion product
$$M^{\Lambda_1,\Lambda_1+\beta_i}\boxtimes M^{\Lambda_2,\Lambda_2+\beta_j}=\sum_{\Lambda_3\in P_+^k}N_{\Lambda_1,\Lambda_2}^{\Lambda_3}M^{\Lambda_3, \Lambda_1+\Lambda_2+\beta_i+\beta_j}$$
is obtained for $\Lambda_1,\Lambda_2\in P_+^k$ and $i,j\in Q/kQ_L$ where $N_{\Lambda_1,\Lambda_2}^{\Lambda_3}$
are the fusion rules for irreducible $L_{\wg}(k,0)$-modules:
$$L_{\wg}(k,\Lambda_1)\boxtimes L_{\wg}(k,\Lambda_2)=\sum_{\Lambda_3\in P_+^k}N_{\Lambda_1,\Lambda_2}^{\Lambda_3}L_{\wg}(k,\Lambda_3).$$
The fusion rules in \cite{CL} and \cite{DW4} are computed in the same fashion.

The paper is organized as follows. We review the basics on quantum dimensions from the theory of vertex operator algebras \cite{DJX1} and  the Frobenius-Perron dimensions from the fusion tensor category \cite{ENO}
and discuss their properties and connections in Section 2. We recall the construction of the parafermion vertex operator algebras $K(\g,k)$ and their representations \cite{DR} in Section 3. We also give an elementary result on the
weight lattice $P$ and the dual lattice $Q_L^{\circ}$ for a simple Lie algebra $\g.$ The Section 4 is on the computation of quantum dimensions of the irreducible $K(\g,k)$-modules. We finish the identification of the irreducible $K(\g,k)$-modules and the determination of the fusion rules in Section 5.

\section{The Frobenius-Perron dimensions and the quantum dimensions}

In this section we review the basic properties of the Frobenius-Perron dimensions from the fusion category and
the quantum dimensions from the vertex operator algebras. In the case the fusion category is the module category for  a rational, $C_2$-cofinite vertex operator algebra, we discuss the connection between these dimensions.

 We first collect basics of the fusion categories and the Frobenius-Perron dimensions from \cite{BK}, \cite{ENO} and \cite{DMNO}.

Let $\cc$ be a fusion category \cite{BK}. That is, $\cc$ is
 a semisimple  rigid monoidal category with finite dimensional spaces of morphisms,
finitely many irreducible objects and an irreducible unit object $\1_{\cc}.$
We use  $K(\cc)$ to denote the Grothendieck ring of $\cc.$  According to \cite{ENO}
we have
\begin{thm}\label{ENO}
There exists
a unique ring homomorphism $\FPdim : K(\cc)\to \R$ such that $\FPdim_{\cc}(X) > 0$ for all
$X\in\cc.$
\end{thm}

The $\FPdim_{\cc}(X)$ is called the Frobenius-Perron dimension of $X\in\cc.$
One can also define the Frobenius-Perron dimension for the category $\cc:$
$$ \FPdim (\cc)=\sum_{X\in {\mathcal O}(\cc)}\FPdim_{\cc}(X)^2$$
where ${\mathcal O}(\cc)$ is the equivalence classes of the simple objects in $\cc.$

An {\em algebra} in a monoidal category $\mathcal{C}$ is an object $A\in \mathcal{C}$ which is an associative
algebra \cite{O}. Let  $\cc_A$ be the left $A$-module category. An algebra $A\in \cc$ is said to be \'{e}tale if it is commutative and the  $\cc_A$ is semisimple.  We say that an \'{e}tale algebra $A$  is connected if
$\dim\Hom_{\cc} (\1, A) = 1.$  Let $\cc_A^0$ be the subcategory of $\cc_A$ consists of dyslectic modules \cite{DMNO}.
The following result from \cite{ENO} and \cite{DMNO} is important in this paper.
\begin{thm}\label{DMNO}
Let $\cc$ be a fusion category and  $A $ be a connected \'{e}tale algebra
in $\cc$. Then
$$(\FPdim_\cc A) \FPdim \cc_A = \FPdim \cc$$
and
$$(\FPdim_\cc A)^2 \FPdim \cc_A^0 = \FPdim \cc$$
\end{thm}

We now turn to the theory of vertex operator algebra. Let $V=(V,Y,\1,\omega)$ be a vertex operator algebra (cf. \cite{B} and \cite{FLM}). Here are some basics on vertex
operator algebras.

The $V$ is of {\em CFT type} if $V$ is simple,  $V=\oplus_{n\geq 0}V_n$ and $V_0=\C {\bf 1}$ \cite{DLMM}.

The $V$ is called $C_2$-cofinite if $\dim V/C_2(V)<\infty$ where
$C_2(V)$ is the subspace of $V$ spanned by $u_{-2}v$ for $u,v\in V$ \cite{Z}.

A {\em weak} $V$-module  $M=(M,Y_M)$ is a vector space equipped with a linear map
$$\begin{array}{l}
V\to (\mbox{End}\,M)[[z^{-1},z]]\label{map}\\
v\mapsto\displaystyle{Y_M(v,z)=\sum_{n\in \Z}v_nz^{-n-1}\ \ (v_n\in\mbox{End}\,M)}
\mbox{ for }v\in V\label{1/2}
\end{array}$$
satisfying the following conditions for $u,v\in V$,
$w\in M$:
\begin{eqnarray*}\label{e2.1}
& &v_nw=0\ \ \  				
\mbox{for}\ \ \ n\in \Z \ \ \mbox{sufficiently\ large};\label{vlw0}\\
& &Y_M({\bf 1},z)=1;\label{vacuum}
\end{eqnarray*}
 \begin{equation*}\label{jacobi}
\begin{array}{c}
\displaystyle{z^{-1}_0\delta\left(\frac{z_1-z_2}{z_0}\right)
Y_M(u,z_1)Y_M(v,z_2)-z^{-1}_0\delta\left(\frac{z_2-z_1}{-z_0}\right)
Y_M(v,z_2)Y_M(u,z_1)}\\
\displaystyle{=z_2^{-1}\delta\left(\frac{z_1-z_0}{z_2}\right)
Y_M(Y(u,z_0)v,z_2)}.
\end{array}
\end{equation*}

An ({\em ordinary}) $V$-module is a  weak $V$-module $M$ which
is $\C$-graded
$$M=\bigoplus_{\lambda \in{\C}}M_{\lambda} $$
such that $\dim M_{\l}$ is finite and $M_{\l+n}=0$
for fixed $\l$ and $n\in {\Z}$ small enough, where
$M_{\l}$ is the eigenspace for $L(0)$ with eigenvalue $\lambda:$
$$L(0)w=\l w=(\mbox{wt}\,w)w, \ \ \ w\in M_{\l}.$$

An {\em admissible} $V$-module is
a  weak $V$-module $M$ which  carries a
${\Z}_{+}$-grading
$$M=\bigoplus_{n\in {\Z}_{+}}M(n)$$
($\Z_+$ is the set all nonnegative integers) such that if $r, m\in {\Z} ,n\in {\Z}_{+}$ and $a\in V_{r}$
then
$$a_{m}M(n)\subseteq M(r+n-m-1).$$
Note that any ordinary module is an admissible module.

A vertex operator algebra $V$ is called {\em rational} if any admissible
module is a direct sum of irreducible admissible modules \cite{DLM1}. It was proved in
\cite{DLM2} that if $V$ is rational then there are only
finitely many inequivalent irreducible admissible modules $V=M^0,...,M^p$
and each irreducible admissible module is an ordinary module. Each $M^i$ has weight space decomposition
$$M^i=\bigoplus_{n\geq 0}M^i_{\lambda_i+n}$$
where $\lambda_i\in\C$ is a complex number such that $M^i_{\lambda_i}\ne 0$ and $M^i_{\lambda_i+n}$ is the eigenspace of $L(0)$ with eigenvalue $\lambda_i+n.$ The $\lambda_i$ is called the conformal weight of $M^i.$
If $V$ is both rational and
$C_2$-cofinite, then $\lambda_i$ and central charge $c$ are rational numbers \cite{DLM3}.

In the rest of this  paper we assume the following:
\begin{enumerate}
\item[(V1)] $V=\oplus_{n\geq 0}V_n$ is a vertex operator algebra  of CFT type,
\item[(V2)] $V$ is $C_{2}$-cofinite and  rational,
\item[(V3)] The conformal weight $\lambda_i$ is nonnegative and $\lambda_i=0$ if and only if $i=0.$
\end{enumerate}

Let $M = \bigoplus_{\lambda\in \mathbb{C}}{M_{\lambda}}$ be a $V$-module. Set $M'
= \bigoplus_{\lambda \in \mathbb{C}}{M_\lambda^*}$, the restricted
dual of $M$. It is proved in \cite{FHL} that $M'=(M', Y')$ is naturally a
$V$-module such that $$\langle Y'(a, z)u', v\rangle  = \langle u', Y(e^{zL(1)}(-z^{-2})^{L(0)}a, z^{-1})v\rangle ,$$\\for $a\in V, u'\in
M'$ and $v\in M,$ and $(M')'\cong M$. Moreover, if $M$ is irreducible, so
is $M'.$ A $V$-module $M$ is said to be self dual if $M$ and $M'$
are isomorphic.

Recall from \cite{FHL} the  notion of intertwining operator and fusion
rule. Let $W^i=(W^i,Y_{W^i})$ for $i=1,2,3$ be  $V$-modules.
An intertwining
operator $\mathcal {Y}( \cdot , z)$ of type $\Itw{W^3}{W^1}{W^2}$ is a linear map
$$
\mathcal
{Y}(\cdot, z): W^1\rightarrow \Hom(W^2, W^3)\{z\}, \, v^1\mapsto
\mathcal {Y}(v^1, z) = \sum_{n\in \mathbb{C}}{v_n^1z^{-n-1}}
$$
such that
\begin{enumerate}
\item[(i)] For any $v^1\in W^1, v^2\in W^2$ and $\lambda \in \mathbb{C},
v_{n+\lambda}^1v^2 = 0$ for $n\in \mathbb{Z}$ sufficiently large.
\item[(ii)] For any $a \in V, v^1\in W^1$,
$$z_0^{-1}\delta(\frac{z_1-z_2}{z_0})Y_{W^3}(a, z_1)\mathcal
{Y}(v^1, z_2)-z_0^{-1}\delta(\frac{z_1-z_2}{-z_0})\mathcal{Y}(v^1,
z_2)Y_{W^2}(a, z_1)$$
$$=z_2^{-1}\delta(\frac{z_1-z_0}{z_2})\mathcal{Y}(Y_{W^1}(a, z_0)v^1, z_2).$$
\item[(iii)] For $v^1\in W^1$, $\dfrac{d}{dz}\mathcal{Y}(v^1,
z)=\mathcal{Y}(L(-1)v^1, z)$.
\end{enumerate}
The intertwining operators of type $\Itw{W^3}{W^1}{W^2}$ form a vector space denoted by $$I_V\Itw{W^3}{W^1}{W^2}.$$
 The dimension $N^{W^3}_{W^1, W^2}$ of $I_V\Itw{W^3}{W^1}{W^2}$ is called the
fusion rule of type $\Itw{W^3}{W^1}{W^2}.$
It is proved in \cite{ABD} that the fusion rules for three irreducible modules are finite.

The following two propositions can be found in \cite{ADL}.
\begin{prop}\label{adl1}
Let $V$ be a vertex operator algebra and let $W^1,\,W^2,\, W^3$ be
$V$-modules among which $W^1$ and $W^2$ are irreducible.
Suppose that $U$ is a vertex operator subalgebra
of $V$ (with the same Virasoro element) and that
$N^1$ and $N^2$ are irreducible $U$-submodules of
$W^1$ and $W^2$, respectively.
Then the restriction map from $I_V\fusion{W^1}{W^2}{W^3}$
to $I_U\fusion{N^1}{N^2}{M^3}$
is injective. In particular,
$$\dim I_V\fusion{W^1}{W^2}{W^3}\leq\dim I_U\fusion{N^1}{N^2}{W^3}.$$
\end{prop}

Let $V^{1}$ and $V^{2}$ be vertex operator algebras,
let $W^{i}\,(i=1,2,3)$ be $V^{1}$-modules and
let $N^{i}\,(i=1,2,3)$ be $V^{2}$-modules. Then $W^i\otimes N^i\, (i=1,2,3)$ are $V^1\otimes V^2$-modules
\cite{FHL}.

\begin{prop}\label{adl2} If $N_{W^1,W^2}^{W^3}<\infty$ or $N_{N^1,N^2}^{N^3}<\infty$ then
$$N_{W^1\otimes N^1,W^2\otimes N^2}^{W^3\otimes N^3}=N_{W^1,W^2}^{W^3}N_{N^1,N^2}^{N^3}.$$
\end{prop}

Let $W^1$ and $W^2$ be two $V$-modules. A tensor product for the
ordered pair $(W^1, W^2)$ is a pair $(W, F(\cdot, z))$, which consists of
a $V$-module $W$ and an intertwining operator $F(\cdot, z)$ of type
$\Itw{W}{W^1}{W^2}$, such that the following universal property
holds: For any $V$-module $M$ and any intertwining operator $I(\cdot,
z)$ of type $\Itw{M}{W^1}{W^2}$, there exists a unique $V$-homomorphism $\phi$
from $W$ to $M$ such that $I(\cdot, z)=\phi\circ F(\cdot, z)$.
It is clear from the definition that if a tensor product of  $W^1$ and $W^2$ exsits, it is unique
up to isomorphism. In this case, we denote the tensor product by $W^1\boxtimes W^2.$

The following results are obtained in \cite{HL1}, \cite{HL2}, \cite{HL3} and \cite{H1}, \cite{H2}, \cite{H3}.
\begin{thm}\label{H} Let $V$ be vertex operator algebra satisfying conditions (V1)-(V3).

1. The tensor product of any two $V$-modules $M\boxtimes N$ exits. In particular,
$M^i\boxtimes M^j$ of $M^i$ and $M^j$ exists and
is equal to $\sum_{k}N_{i,j}^kM^k$ for any $i,j\in\{0,...,p\}.$

2. The $V$-module category $\CC_V$ is a fusion category.
\end{thm}

The modular transformation of trace functions of irreducible modules of vertex operator algebra \cite{Z} is another important ingredient in this paper. It is defined another
 VOA structure $(V,Y[\cdot,z],\1,\omega-c/24)$ on $V$ in \cite{Z} with grading
 $$V=\bigoplus_{n\geq 0}V_{[n]}.$$
 For $v\in  V_{[n]}$ we write $wt [v]=n.$ We denote $v_{n-1}$ by $o(v)$ for $v\in V_n$
 and  extend to $V$ linearly. For $i=0,...,p$ and $v\in V,$ we set
 $$Z_i(v,q)=\tr_{M^i}o(v)q^{L(0)-c/24}=\sum_{n\geq 0}(\tr_{M^i_{\lambda_i+n}}o(v))q^{\lambda_i+n-c/24}$$
 which is a formal power series in variable $q.$ The constant $c$ here is the central charge of $V.$
 The $Z_i(\1,q)$ which is denoted by $\ch_qM^i$ sometimes is called the $q$-character of $M^i.$ The $Z_i(v, q)$ converges to a holomorphic function
 in $0<|q|<1$ \cite{Z}. Let $\H=\{\tau\in \C \mid im\,\tau>0\}$ be the upper half complex plane  and
 $q=e^{2\pi i\tau}$ with $\tau\in \H.$ Denote by the  $Z_i(v,\tau)$ the holomorphic function $Z_i(v,q)$
 on $\H.$

 Note that the modular group $SL_2{(\mathbb{Z})}$ acts on $\H$ in an obvious way.
\begin{thm}\label{2.3} Let $V$ be a vertex operator algebra satisfying (V1)-(V3).

(1) There is  a group homomorphism $\rho: SL_2(\Z)\to GL_{p+1}(\C)$ with $\rho(\gamma)=(\gamma_{ij})$
such that for any $0\leq i\leq p$ and $v\in V_{[n]}$
$$Z_i(v,\gamma\tau)=({c\tau+d})^n\sum_{j=0}^p\gamma_{ij}Z_j(v,\tau).$$

(2) Each $Z_i(v,\gamma\tau)$ is a modular form of weight $n$ over a congruence subgroup $\Gamma(m)$ for some $m\geq 1.$
 \end{thm}

Part (1) of the Theorem was obtained in \cite{Z} and Part (2) was established in \cite{DLN}.
The
$$S=\rho_V\left(\mtx{0 & -1\\ 1 & 0}\right) \text{ and }T=\rho_V\left(\mtx{1 & 1\\ 0 & 1}\right).
$$
are respectively called the \emph{genus one} $S$ and $T$-matrices of $V$.

Finally we can define the quantum dimension. Let $V$ be as before and $M$ be a $V$-module. Then $M=\sum_{i=0}^pM^i$ is a direct sum of finitely many irreducible $V$-modules.  Then both
$Z_V(\tau)=Z_0(\1,\tau)$ and $Z_M(\tau)=Z_M(1,\tau)$ exist.
The quantum dimension of $M$ over $V$ is defined as
$$
\qdim_{V}M=\lim_{y\to 0}\frac{Z_M(iy)}{Z_V(iy)}=\lim_{q\to 1}\frac{\ch_qM}{\ch_qV}$$
where $y$ is real and positive and $\ch_qM$ is the $q$-character of $M$.

 The following result was given in \cite{DJX1}.
\begin{thm}\label{DJX1}
Let $V$ be a vertex operator algebra satisfying (V1)-(V3).

(1) $\qdim_VM^i=\frac{S_{i0}}{S_{00}}$ exits and is greater than or equal to $1$ for all $i$
where $S=(S_{ij}).$

(2) $\qdim_VM^i$ is the maximal eigenvalue of the fusion matrix $N(i)=(N_{ij}^k)_{jk}.$

(3) $\qdim_VM^i\boxtimes M^j=\qdim_VM^i\cdot \qdim_VM^j$ for all $i,j.$

(4) $M^i$ is a simple current if and only if $\qdim_VM^i=1.$
\end{thm}

By Theorems \ref{ENO}, \ref{H} and \ref{DJX1}, we see the relation between the quantum dimension and the Frobenius-Perron dimension: $\qdim_VM^i=\FPdim_{\cc_V}(M^i)$ for all $i.$ We also define the global dimension
$${\rm glob}(V)=\sum_{i=0}^p(\qdim_VM^i)^2.$$
It is clear that ${\rm glob}(V)=\FPdim(\cc_V).$

An extension $U$ of $V$ is a simple vertex operator algebra containing $V.$ Then $U$ is again $C_2$-cofinite \cite{ABD}. Here we quote a recent result from \cite{HKL}.
\begin{thm}\label{HKL} Let $V$ be a vertex operator algebra satisfying (V1)-(V3).

(1) If $U$ is an extension vertex operator algebra of $V$, then
$U$ induces an \'{e}tale algebra $A_{U}$ in $\cc_V$ such that $A_U$ is isomorphic to $ U$ as $V$-module.

(2) If $U$ is a $V$-module having integral conformal weight and $U$ is a commutative algebra  in $\cc_V$, then $U$ has a vertex operator algebra structure such that $U$ is an extension vertex operator algebra of $V$.

(3) $U$ is rational.
\end{thm}

\begin{thm}\label{DMNO1} Let $V$ be a vertex operator algebra satisfying (V1)-(V3) and simple vertex operator algebra $U$ be an extension of $V.$ Then $U$ also satisfies (V1)-(V3) and
$${{\rm glob}}(V)={{\rm glob}}(U)(\qdim_VU)^2.$$
\end{thm}
\begin{proof} The theorem is a combinations of \cite{KO}, Theorem \ref{DMNO}, \cite{DJX1} and \cite{HKL}. In this case,
let $\cc$ be the category of $V$-modules and $A=U$ an algebra in $\cc.$ Then $\cc_A^0$ is the $U$-module category by \cite{KO}.
\end{proof}

We now consider two vertex operator algebras $V^1,V^2$ satisfying conditions (V1)-(V3). Then it is easy to see that the tensor product vertex operator algebra $V^1\otimes V^2$ \cite{FHL} also satisfies assumptions (V1)-(V3).
\begin{lem}\label{qtensor} Let $M$ be a $V^1$-module and $N$ be a $V^2$-module. Then
$$\qdim_{V^1\otimes V^2}M\otimes N=\qdim_{V^1}M\cdot \qdim_{V^2}N,$$
$${\rm glob}(V^1\otimes V^2)={\rm glob}(V^1){\rm glob}(V^2).$$
\end{lem}

\begin{proof} The equality $\qdim_{V^1\otimes V^2}M\otimes N=\qdim_{V^1}M\cdot \qdim_{V^2}N$ follows
from the fact that $Z_{M\otimes N}(\tau)=Z_M(\tau)Z_N(\tau)$ and  the equality  ${\rm glob}(V^1\otimes V^2)={\rm glob}(V^1){\rm glob}(V^2)$ follows from the first equality and the fact that the irreducible $V^1\otimes V^2$-modules are exactly
$M\otimes N$ where $M$ is an irreducible $V^1$-module and $N$ is an irreducible $V^2$-module \cite{FHL}.
\end{proof}

\section{Parafermion vertex operator algebras}

In this section we recall the parafermion vertex operator algebra $K(\g,k)$ and its representations
associated to any finite dimensional simple Lie algebra $\g$ and positive integer $k$ from \cite{DR}.

Let $\g$ be a finite dimensional simple Lie algebra  with a Cartan
subalgebra $\h.$ We denote the corresponding root system by $\Delta$
and the root lattice by $Q.$ Fix an invariant symmetric
nondegenerate bilinear form $\< ,\>$  on $\g$ such that $\<\a,\a\>=2$ if
$\alpha$ is a long root, where we have identified $\h$ with $\h^*$
via $\<,\>.$ We denote the image of $\alpha\in
\h^*$ in $\h$ by $t_\alpha.$ That is, $\alpha(h)=\<t_\alpha,h\>$
for any $h\in\h.$ Fix simple roots $\{\alpha_1,...,\alpha_l\}$
and let $\Delta_+$ be the set of corresponding positive roots.
Denote the highest root by $\theta.$

Recall that the weight lattice $P$ of $\g$ consists of $\lambda\in \h^*$ such that
$\frac{2\<\lambda,\alpha\>}{\<\alpha,\alpha\>}\in\Z$ for all $\alpha\in \Delta.$ It is well-known that
$P=\bigoplus_{i=1}^l\Z\Lambda_i$ where $\Lambda_i$ are the fundamental weights defined by the equation
$\frac{2\<\Lambda_i,\alpha_j\>}{\<\alpha_j,\alpha_j\>}=\delta_{i,j}.$ Let $P_+$ be the subset
of $P$ consisting of the dominant weight $\Lambda\in P$ in the sense that $\frac{2\<\Lambda,\alpha_j\>}{\<\alpha_j,\alpha_j\>}$ is nonnegative for all $j.$ For any nonnegative
integer $k$ we also let $P_+^k$ be the subset of $P_+$ consisting of $\Lambda$ satisfying $\<\Lambda,\theta\>
\leq k.$

Let $Q=\sum_{i=1}^l\Z\alpha_i$ be the root lattice and $Q_L$ be the sublattice of $Q$ spanned by the long roots. Recall that the dual lattice $Q_L^{\circ}$ consists $\lambda\in \h^*$ such that $\<\lambda,\alpha\>\in\Z$ for
all $\alpha\in Q_L.$

For the purpose of identifying the irreducible $K(\g,k)$-module, we need the following Lemma.
\begin{lem}\label{QLP} For any simple Lie algebra $\g,$ $Q_L^{\circ}=P.$
\end{lem}
\begin{proof} The result is obvious if $\g$ is a Lie algebra of type $A, D,E$ as
 $\<\alpha,\alpha\>=2$ if $\alpha$ is a long root. We now assume that $\g$ is a Lie algebra of  other type. First, observe that $P\subset Q_L^{\circ}.$ It remains to show that $Q_L^{\circ}\subset P.$ We will do a verification case by case using the root systems given in \cite{H}.

 (1) Type $B_l.$ Let $\E=\R^l$ with the standard orthonormal basis $\{\epsilon_1,...,\epsilon_l\}.$ Then
 $$\Delta=\{\pm\epsilon_i, \pm(\epsilon_i\pm\epsilon_j)|i\ne j\}.$$
 Let $\lambda\in Q_L^{\circ}.$ Then $\<\lambda,(\epsilon_i\pm\epsilon_j)\>\in\Z$ for all $i\ne j.$ This implies
 that $\<\lambda,\epsilon_i\>\in\frac{1}{2}\Z$ and $\frac{2\<\lambda,\alpha\>}{\<\alpha,\alpha\>}\in\Z$
 if $\alpha$ is a short root. That is, $\lambda\in P.$

 (2) Type $C_l.$ In this case, $$\Delta=\{\pm\sqrt{2}\epsilon_i, \pm\frac{1}{\sqrt{2}}(\epsilon_i\pm\epsilon_j)|i\ne j\}.$$
 If $\lambda\in Q_L^{\circ},$ then $\<\lambda,\epsilon_i\>\in\frac{1}{\sqrt{2}}\Z$ and $\frac{2\<\lambda,\alpha\>}{\<\alpha,\alpha\>}\in\Z$
 if $\alpha$ is a short root.

 (3) Type $F_4.$ Let $\E=\R^4.$ Then
 $$\Delta=\{\pm\epsilon_i, \pm(\epsilon_i\pm\epsilon_j), \pm\frac{1}{2}(\epsilon_1\pm \epsilon_2\pm \epsilon_3\pm \epsilon_4)|i\ne j\}.$$
 If $\lambda\in Q_L^{\circ},$ then $\<\lambda,\epsilon_i\>\in\frac{1}{2}\Z$ and $\frac{2\<\lambda,\alpha\>}{\<\alpha,\alpha\>}\in\Z$
 if $\alpha$ is a short root.

 (4) Type $G_2.$ Let $\E$ be the subspace of $\R^3$ orthogonal to $\epsilon_1+\epsilon_2+\epsilon_3.$
 Then
 $$\Delta=\pm\frac{1}{\sqrt{3}}\{\epsilon_i-\epsilon_j, 2\epsilon_1-\epsilon_2-\epsilon_3, 2\epsilon_2-\epsilon_1-\epsilon_3,
  2\epsilon_3-\epsilon_1-\epsilon_2|i\ne j\}.$$
 If $\lambda\in Q_L^{\circ},$ then
 $$\<\lambda, \frac{1}{\sqrt{3}}(2\epsilon_1-\epsilon_2-\epsilon_3- 2\epsilon_2+\epsilon_1+\epsilon_3)\>\in \Z.$$
 This gives $\<\lambda, \epsilon_1-\epsilon_2\>\in \frac{1}{\sqrt{3}}\Z$ and $\frac{2\<\lambda, \frac{1}{\sqrt{3}}(\epsilon_1-\epsilon_2)\>}{\frac{1}{3}\<\epsilon_1-\epsilon_2,\epsilon_1-\epsilon_2\>}\in\Z.$
 Similarly, one can verify that $\frac{2\<\lambda,\alpha\>}{\<\alpha,\alpha\>}\in\Z$
 for any short root $\alpha.$ The proof is complete.
\end{proof}

Let $\wg=\g\otimes \C[t,t^{-1}]\oplus \C K$ be the affine Lie algebra. Fix a nonnegative integer $k.$ For any $\Lambda\in P_+^k$ let $L(\Lambda)$ be the irreducible highest weight $\g$-module with highest weight $\Lambda$
and $L_{\wg}(k,\Lambda)$ be the unique irreducible $\wg$-module such that $L_{\wg}(k,\Lambda)$ is generated by $L(\Lambda)$ and $\g\otimes t^n L(\Lambda)=0$ for $t>0$ and $K$ acts as constant $k.$ The following result is well known
(cf. \cite{FZ},\cite{LL}, \cite{Z}):
\begin{thm} The $L_{\wg}(k,0)$ is a vertex operator algebra satisfying conditions (V1)-(V3). Namely,
$L_{\wg}(k,0)$ is a simple, rational and $C_2$-cofinite  vertex operator algebra whose irreducible modules
are $L_{\wg}(k,\Lambda)$ for $\Lambda\in P_+^k$ and the weight $\lambda_{L_{\wg}(k,\Lambda)}$ of $L_{\wg}(k,\Lambda)$ is
$\frac{\<\Lambda+2\rho, \Lambda\>}{2(k+h^{\vee})}$ where  $\rho=\sum_{i=1}^l\Lambda_i$ and
$h^{\vee}$ is the dual Coxeter number.
\end{thm}

Let $\theta=\sum_{i=1}^la_i\alpha_i.$ Here is a  list of $a_i=1$  using the labeling from \cite{K}:
\begin{eqnarray*}
A_l: & a_1,...,a_l\\
B_l: & a_1\\
C_l: &a_l\\
D_l: &a_1, a_{l-1}, a_l\\
E_6: &a_1, a_5\\
E_7: &a_6
\end{eqnarray*}
Denote by $I$ the set of $i$ with $a_i=1.$ It is easy to see that the cardinality of $I$ is equal to $|P/Q|$ \cite{L3}.

Let $M_{\widehat{\h}}(k)$ be the vertex operator subalgebra of $L_{\widehat{\g}}(k,0)$
generated by $h(-1)\1$ for $h\in \mathfrak h.$
For $\lambda\in
{\mathfrak h}^*,$ denote by  $M_{\widehat{\h}}(k,\lambda)$ the irreducible
highest weight module for $\wh$ with a highest weight vector
$e^\lambda$ such that $h(0)e^\lambda = \lambda(h) e^\lambda$ for
$h\in \mathfrak h.$ The parafermion vertex operator algebra $K(\g,k)$ is the commutant \cite{FZ} of
$M_{\widehat{\h}}(k)$ in $L_{\wg}(k,0).$ We have the following decomposition
$$L_{\wg}(k,\Lambda)=\bigoplus_{\lambda\in Q+\Lambda}M_{\widehat{\h}}(k,\lambda)\otimes M^{\Lambda,\lambda}$$
as $M_{\widehat{\h}}(k)\otimes K(\g,k)$-module. Moreover, $M^{0,0}=K(\g,k)$ and $M^{\Lambda,\lambda}$ is an
irreducible $K(\g,k)$-module \cite{DR}.

It is proved in \cite{DW3} that the lattice vertex operator algebra $V_{\sqrt{k}Q_L}$ is a vertex operator subalgebra of $L_{\widehat{\g}}(k,0)$ and the parafermion vertex operator algebra $K(\g,k)$ is also a commutant
of $V_{\sqrt{k}Q_L}$ in $L_{\widehat{\g}}(k,0).$ This gives us another decomposition
$$L_{\widehat{\g}}(k,\Lambda)=\bigoplus_{i\in Q/kQ_L}V_{\sqrt{k}Q_L+\frac{1}{\sqrt{k}}(\Lambda+\beta_i)}\otimes M^{\Lambda,\Lambda+\beta_i}$$
as modules for $V_{\sqrt{k}Q_L}\otimes K(\g,k)$ where $M^{\Lambda,\lambda}$ is as before and
$Q=\cup_{i\in Q/kQ_L} (kQ_L+\beta_i).$

Here are the main results on $K(\g,k).$
\begin{thm}\label{DR} Let $\g$ be a simple Lie algebra and $k$ a positive integer.

(1) The $K(\g,k)$ is a vertex operator algebra satisfies conditions (V1)-(V3).

(2) For any $\Lambda\in P_+^k,$ $\lambda\in \Lambda+Q$
and $\alpha\in Q_L,$ $M^{\Lambda,\lambda}=M^{\Lambda,\lambda+k\alpha}.$

(3) For each $i\in I,$ $\Lambda\in P_+^k$ there exists a unique  $\Lambda^{(i)}\in P_+^k$  such that
for any $\lambda\in \Lambda+Q,$ $M^{\Lambda,\lambda}=M^{\Lambda^{(i)},\lambda+k\Lambda_i}.$

(4) Any irreducible $K(\g,k)$-module is isomorphic to $M^{\Lambda,\lambda}$ for some $\Lambda\in P_+^k$
and $\lambda\in \Lambda+Q.$
\end{thm}

The $C_2$-cofiniteness of $K(\g,k)$ was obtained in \cite{ALY1} (also see \cite{M}) and the rest results in Theorem can be found in \cite{ALY2} and \cite{DR}.

The following result will be useful later.
\begin{lem}\label{iden3}
 Fix $\Lambda\in P_+^k$ and $\lambda\in \Lambda+Q.$ Let $A=\{\Lambda+\beta_j+kQ_L|j\in Q/kQ_L\}.$
Then the set $\{(\Lambda, \lambda+kQ_L), (\Lambda^{(i)}, \lambda+k\Lambda_i+kQ_L) |i\in I\}$ gives  exactly $|I|$ elements
in $P_+^k\times A.$
\end{lem}
\begin{proof} It is proved in \cite{DR} that $(\Lambda, \lambda+kQ_L)$
is different from $(\Lambda^{(i)}, \lambda+k\Lambda_i+kQ_L)$ for $i\in I.$ Let $i,j\in I$ be distinct. We can assume that $\Lambda^{(i)}=\Lambda^{(j)}.$ Then
$$(\Lambda^{(i)}, \lambda+k\Lambda_i+kQ_L)= (\Lambda^{(j)}, \lambda+k\Lambda_j+kQ_L)$$
if and only if $\Lambda_i-\Lambda_j\in Q_L.$ If $\g$ is of $A,D,E$ type, this cannot happen. For type $B_l$ and
$C_l$ the cardinality of $I$ is 1. The proof is complete.
\end{proof}

From Theorem \ref{DR} and Lemma \ref{iden3} we immediately have:
\begin{coro} Let $\g$ be a simple Lie algebra and $k$ a positive integer. Then
$K(\g,k)$ has at most $\frac{|P_+^k||Q/kQ_L|}{|P/Q|}$ inequivalent irreducible modules.
\end{coro}

\section{Quantum dimensions of the parafermion vertex operator algebras}

In this section we compute the quantum dimensions of irreducible $K(\g,k)$-modules. The ideas and methods here
are different from these used in \cite{DW2}. We do not need the $S$-matrix for the computation.

First we need a result on the quantum dimension in orbifold theory from \cite{DJX1}. Let $V$ be a simple
vertex operator algebra and $G$ a finite automorphism group of $V.$ Then $V^G$ is a vertex operator subalgebra
and $V$ has a decomposition
$$V=\bigoplus_{\chi\in \hat G}W_{\chi}\otimes V_{\chi}$$
where $\hat G$ is the set of irreducible characters of $G$ and $W_{\chi}$ is the simple $G$-module with the character $\chi$ and $V_{\chi}$ is an irreducible $V^G$-module \cite{DM}, \cite{DLM0}. We need the following result
from \cite{DJX1}.
\begin{thm}\label{quan}
Let $V$ be  a vertex operator aglebra satisfying (V1)-(V3) and $G$ a finite automorphism group of $V$
such that $V$ is $g$-rational for every $g\in G$ and any irreducible $g$-twisted $V$-module
$M=\oplus_{n\geq 0}M_{\lambda +\frac{n}{T}}$ has positive conformal weight $\lambda$ if $g \neq 1$ where $T$ is the order of $g.$
Then $\qdim_{V^G} V_{\chi}= \dim W_{\chi}.$
\end{thm}

The next result tells us how the rationality of $V^G$ implies the $g$-rationality of all $ g\in G.$
\begin{lem}\label{general} Let $V$ be a simple vertex operator algebra and $G$ a finite automorphism group of $V$
such that both $V$ and $V^G$ satisfies assumptions (V1)-(V3). Then $V$ is $g$-rational for all $g\in G$
and the conformal weight of any irreducible $g$-twisted $V$-module is positive.
\end{lem}
\begin{proof} Since $V^G$ satisfies assumptions (V1)-(V3), the vertex operator subalgebra $V^{\<g\>}$ of $V$
is $C_2$-cofinite \cite{ABD} and rational \cite{HKL}. Here we need some facts about associative algebra $A_{g,n}(V)$ for $g\in G$ and $0\leq n\in \frac{1}{T}\Z$  from \cite{DLM4}, \cite{DLM5} where $T$ is the order
of $g.$ Let $A_m(V)=A_{1,m}(V).$ The following are true: (1) $V$ is $g$-rational if and only if $A_{g,n}(V)$
is semisimple for all $n,$ (2) There is an onto algebra homomorphism from $A_{[n]}(V^{\<g\>})$ to $A_{g,n}(V)$
where $[n]$ is the largest integer less than or equal to $n.$ Since $V^{\<g\>}$ is rational, $A_m(V^{\<g\>})$ is semisimple for all $m$. Thus, $A_{g,n}(V)$ is semisimple for
all $n$ and $V$ is $g$-rational.

It remains to prove that the conformal weight $\lambda$ of any irreducible $g$-twisted $V$-module $M$ is positive if $g\ne 1.$ Let $T$ be the order of $g.$ Then $M$ has decomposition
$$M=\bigoplus_{n=0}^{\infty}M_{\lambda+\frac{n}{T}}$$
such that $M_{\lambda}\ne 0$ where $ M_{\lambda+\frac{n}{T}}$ is the eigenspace of $L(0)$ with eigenvalue $\lambda+\frac{n}{T}$ \cite{DLM2}. Moreover, for each $s=0,...,T-1$ the subspace $\bigoplus_{n\in\Z}M_{\lambda+n+\frac{s}{T}}\ne 0.$ If $\lambda=0$ then there exists an irreducible
$V^G$-submodule $W$ of $M$ isomorphic to $V^G$ as $V^G$ is the only irreducible $V^G$-module whose conformal weight is $0.$ Let $V=\bigoplus_{s}V^s$ where $V^s$ is the irreducible $V^G$-module. Then
$M=\sum_{s}V^s\cdot W$ where $V^s\cdot W$ is a subspace spanned by $u_nW$ for $u\in V^s$ and $n\in\Z.$
It is easy to see that  $V^s\cdot W$ is isomorphic to $V^s$ as $V^G$-module. This implies that $M$ has only integral weights. This is a contradiction as $T\ne 1.$
\end{proof}

Recall the irreducible $K(\g,k)$-module $M^{0,\beta_i}$ for $i\in Q/kQ_L$ from Section 3.
\begin{lem}\label{qtensor1}
The $M^{0,\beta_i}$ is a simple current.
\end{lem}

\begin{proof}   We need to introduce a finite abelain group $G$ following \cite{DR}.  Let $G$ be the dual group of the finite abelian group $Q/kQ_L.$ Then $G$ is a group
of automorphisms of  $L_{\wg}(k,0)$ such that $g\in G$ acts as $g(\beta_i+kQ_L)$ on $V_{\sqrt{k}Q_L+\frac{1}{\sqrt{k}}\beta_i}\otimes M^{0,\beta_i}.$  Clearly, each $\beta_i+kQ_L$
is an irreducible character of $G.$  So $V_{\sqrt{k}Q_L+\frac{1}{\sqrt{k}}\beta_i}\otimes M^{0, \beta_i}$ in the decomposition
$$L_{\wg}(k,0)=\bigoplus_{i\in Q/kQ_L}V_{\sqrt{k}Q_L+\frac{1}{\sqrt{k}}\beta_i}\otimes M^{0, \beta_i}$$
corresponds to the character $\beta_i+kQ_L.$
In particular, $L_{\wg}(k,0)^G=V_{\sqrt{k}Q_L}\otimes K(\g,k).$

It follows from Theorem \ref{quan} and Lemma \ref{general},  $\qdim_{V_{\sqrt{k}Q_L}\otimes K(\g,k)}V_{\sqrt{k}Q_L+\frac{1}{\sqrt{k}}\beta_i}\otimes M^{0, \beta_i}=1.$ We can also use the $g$-rationality
of $L_{\wg}(k,0)$ from \cite{L2}. It is well known that every irreducible $V_{\sqrt{k}Q_L}$-module is a simple current \cite{DL}. Then by Lemma \ref{qtensor},
$\qdim_{K(\g,k)}M^{0, \beta_i}=1.$ Since $K(\g,k)$ satisfies conditions (V1)-(V3). It follows that
$M^{0, \beta_i}$ is a simple current.
\end{proof}

One can also obtain Lemma \ref{qtensor1} by using the mirror extension \cite{Lin}.

The next result asserts that all the irreducible $K(\g,k)$-modules occurring in $L_{\wg}(k,\Lambda)$
for $\Lambda\in P_+^k$ have the same quantum dimension.
\begin{lem}\label{ij}
Let $\Lambda\in P_+^k$ Then $\qdim_{K(g,k)}M^{\Lambda, \lambda}= \qdim_{K(g,k)}M^{\Lambda, \Lambda}$
for all $\lambda\in \Lambda+Q.$
\end{lem}
\begin{proof} By Theorem \ref{DR}, every $M^{\Lambda,\lambda}$ for $\lambda\in \Lambda+Q$ is isomorphic to
$M^{\Lambda,\Lambda+\beta_i}$ for some $i\in Q/kQ_L.$ So it is sufficient to show that all the
$M^{\Lambda,\Lambda+\beta_i}$ have the same quantum dimension.

Recall the decompositions
$$L_{\wg}(k,\Lambda)=\bigoplus_{i\in Q/kQ_L} V_{\sqrt{k}Q_L+\frac{1}{\sqrt{k}}(\Lambda+\beta_i)}\otimes M^{\Lambda, \Lambda+\beta_i}$$
 and
\[
L_{\wg}(k,0)=\bigoplus_{i\in Q/kQ_L}V_{\sqrt{k}Q_L+\frac{1}{\sqrt{k}}\beta_i}\otimes M^{0, \beta_i}.
\]
Since $L_{\wg}(k,\Lambda)$ is an irreducible $L_{\wg}(k,0)$-module,
we see that
 $$L_{\wg}(k,0) \cdot V_{\sqrt{k}Q_L+\frac{1}{\sqrt{k}}(\Lambda+\beta_i)}\otimes M^{\Lambda, \Lambda+\beta_i}=L_{\wg}(k,\Lambda)$$
 for any $i\in Q/kQ_L.$ Here we use $X\cdot W$ to denote the subspace spanned by $u_nW$
 for $u\in X$ and $n\in \Z$ where $X$ is a subspace of a vertex operator algebra and $W$ is a subset
 of a $V$-module. It follows that
\[
\bigoplus_{j\in Q/kQ_L}V_{\sqrt{k}Q_L+\frac{1}{\sqrt{k}}\beta_j}\otimes M^{0, \beta_j}\cdot V_{\sqrt{k}Q_L+\frac{1}{\sqrt{k}}(\Lambda+\beta_i)}\otimes M^{\Lambda, \Lambda+\beta_i}=L_{\wg}(k,\Lambda).
\]
This implies that
$$V_{\sqrt{k}Q_L+\frac{1}{\sqrt{k}}\beta_j}\otimes M^{0, \beta_j}\cdot V_{\sqrt{k}Q_L+\frac{1}{\sqrt{k}}(\Lambda+\beta_i)}\otimes M^{\Lambda, \Lambda+\beta_i}=V_{\sqrt{k}Q_L+\frac{1}{\sqrt{k}}(\Lambda+\beta_i+\beta_j)}\otimes M^{\Lambda, \Lambda+\beta_i+\beta_j}$$
for all $j.$ By Lemma \ref{qtensor1}, $V_{\sqrt{k}Q_L+\frac{1}{\sqrt{k}}\beta_j}\otimes M^{0, \beta_j}$
is a simple current. It follows from Theorem \ref{qtensor} that
$$\qdim_{K(g,k)}M^{\Lambda, \Lambda+\beta_i+\beta_j}=\qdim_{K(g,k)}M^{0, \beta_j} \cdot \qdim_{K(g,k)}M^{\Lambda, \Lambda+\beta_i}=\qdim_{K(g,k)}M^{\Lambda, \Lambda+\beta_i}.$$
The proof is complete.
\end{proof}

We now can give an explicit expression for the quantum dimension of any irreducible $K(\g,k)$-module $M^{\Lambda,\lambda}.$ Recall from \cite{C} that the quantum dimension $
\qdim_{L_{\wg}(k,0)}L_{\wg}(k,\Lambda)=\prod_{\alpha>0}\frac{(\Lambda+\rho,\alpha)_q}{(\rho,\alpha)_q}$
(see Introduction).
\begin{thm}\label{quanequal}
For any $\Lambda \in P^k_+$ and $\lambda\in \Lambda+Q,$
$$\qdim_{K(g,k)}M^{\Lambda, \lambda}= \qdim_{L_{\wg}(k,0)}L_{\wg}(k,\Lambda).$$
\end{thm}

\begin{proof}
The proof is a straightforward computation by noting that the irreducible modules of $V_{\sqrt{k}Q_L}$ are simple currents:
\begin{equation*}
\begin{split}
&\ \ \qdim_{L_{\wg}(k,0)} L_{\wg}(k,\Lambda)=\lim_{q\to 1}\frac{\ch_qL_{\wg}(k,\Lambda)}{\ch_qL_{\wg}(k,0)}\\
& =\lim_{q\to 1}\frac{\sum_{i\in Q/kQ_L} \ch_q V_{\sqrt{k}Q_L+\frac{1}{\sqrt{k}}(\Lambda+\beta_i)}\cdot \ch_q M^{\Lambda, \Lambda+\beta_i}}{\sum_{i\in Q/kQ_L} \ch_q V_{\sqrt{k}Q_L+\frac{1}{\sqrt{k}}\beta_i}\cdot \ch_q M^{0, \beta_i}}\\
&=\lim_{q\to 1}\frac{\big(\sum_{i\in Q/kQ_L} \ch_q V_{\sqrt{k}Q_L+\frac{1}{\sqrt{k}}(\Lambda+\beta_i)}\cdot \ch_q M^{\Lambda, \Lambda+\beta_i}\big) / \big(\ch_q V_{\sqrt{k}Q_L} \cdot \ch_q K(\g,k)\big)}
{\big( \sum_{i\in Q/kQ_L} \ch_q V_{\sqrt{k}Q_L+\frac{1}{\sqrt{k}}\beta_i}\cdot \ch_q M^{0, \beta_i}\big) / \big(\ch_q V_{\sqrt{k}Q_L} \cdot \ch_q K(\g, k)\big)}\\
&=\frac{\sum_{i\in Q/kQ_L}\lim_{q\to 1}\frac{ \ch_q V_{\sqrt{k}Q_L+\frac{1}{\sqrt{k}}(\Lambda+\beta_i)}}{ch_q V_{\sqrt{k}Q_L}}\lim_{q\to 1}\frac{\ch_q M^{\Lambda, \Lambda+\beta_i}}{\ch_q K(\g,k)}}
{\sum_{i\in Q/kQ_L}\lim_{q\to 1}\frac{\ch_q V_{\sqrt{k}Q_L+\frac{1}{\sqrt{k}}\beta_i}}{\ch_q V_{\sqrt{k}Q_L}}\lim_{q\to 1}\frac{\ch_q M^{0, \beta_i}}{\ch_q K(\g, k)}}\\
&=\frac{\sum_{i\in Q/kQ_L}\lim_{q\to 1}\frac{\ch_q M^{\Lambda, \Lambda+\beta_i}}{\ch_q K(\g,k)}} {|Q/kQ_L|}\\
&=\frac{|Q/kQ_L|\cdot \qdim _{K(\g,k)}M^{\Lambda, \Lambda+\beta_i}}{|Q/kQ_L|}\\
&=\qdim _{K(\g,k)}M^{\Lambda, \Lambda+\beta_i}
\end{split}
\end{equation*}
for any $i\in Q/kQ_L.$
\end{proof}

We remark that one can also use the $S$-matrix given in \cite{K} to compute the quantum dimension of $M^{\Lambda,\lambda}$ from the definition. But it will be very complicated as we do not have a complete classification of the irreducible $K(\g,k)$-modules at this point.

For an arbitrary simple vertex operator algebra $V$ and  a finite automorphism group $G$ such that $V^G$ is regular,
the quantum dimensions of irreducible $V^G$-modules are determined in terms of the quantum dimensions
of $g$ twisted modules for $g\in G$ recently in \cite{DRX}. Theorem \ref{quanequal} also follows from these results easily.

\section{Classification of the irreducible modules and the fusion rules}

In this section, we classify the irreducible $K(\g,k)$-modules and determine the fusion rules.

By Theorem \ref{DR}, there are at most $\frac{|P^k_+||Q/kQ_L|}{|P/Q|}$  inequivalent irreducible
$K(\g,k)$-modules.  We prove that there are exactly $\frac{|P^k_+||Q/kQ_L|}{|P/Q|}$ inequivalent irreducible
$K(\g,k)$-modules. Namely, the identification given in \cite{DR} is complete.

\begin{thm}\label{idenm} Let $\g$ be a simple Lie algebra and $k$ a positive integer. Then
there are exactly $\frac{|P^k_+||Q/kQ_L|}{|P/Q|}$ inequivalent irreducible  $K(\g,k)$-modules.
\end{thm}

\begin{proof} Recall that $L_{\wg}(k,0)=\oplus_{i\in Q/kQ_L}V_{\sqrt{k}Q_L+\frac{1}{\sqrt{k}}\beta_i}\otimes M^{0, \beta_i}.$ By Theorems \ref{DMNO1}, \ref{qtensor} and Lemma \ref{qtensor1}  we have
\begin{equation*}
\begin{split}
&  \ \ \ \ {\rm glob} (V_{\sqrt{k}Q_L}\otimes K(\g,k))\\
& ={\rm glob} ( L_{\wg}(k,0)) \cdot \big(\qdim_{V_{\sqrt{k}Q_L}\otimes K(\g,k)}\sum_{i\in Q/kQ_L} V_{\sqrt{k}Q_L+\frac{1}{\sqrt{k}}\beta_i}\otimes M^{0, \beta_i}\big)^2\\
                                                                         &={\rm glob} ( L_{\wg}(k,0))  \big(\sum_{i\in Q/kQ_L}\qdim_{V_{\sqrt{k}Q_L} }V_{\sqrt{k}Q_L+\frac{1}{\sqrt{k}}\beta_i} \cdot \qdim_{K(\g,k)}M^{0, \beta_i} \big)^2\\
                                                                         &={\rm glob} ( L_{\wg}(k,0))  |Q/kQ_L|^2.
\end{split}
\end{equation*}
From Lemma \ref{qtensor} we have
 $${\rm glob} (V_{\sqrt{k}Q_L}) {\rm glob} (K(\g,k))={\rm glob} ( L_{\wg}(k,0))  |Q/kQ_L|^2.$$

We need to determine the global dimension of $V_{\sqrt{k}Q_L}$ first.
Note that  $V_{\sqrt{k}Q_L}$ has exactly $|(\sqrt{k}Q_L)^{\circ}/Q_L|$ inequivalent
irreducible modules \cite{D} and each irreducible is a simple current.
It is evident that $ (\sqrt{k}Q_L)^{\circ}=\frac{1}{\sqrt{k}}Q_L^{\circ}$.
Then
\[
{\rm glob} (V_{\sqrt{k}Q_L}) =\sum_{i\in \frac{1}{\sqrt{k}}Q_L^{\circ}/ \sqrt{k}Q_L}(\qdim_{V_{\sqrt{k}Q_L}}V_{\sqrt{k}Q_L+\lambda_i})^2=|Q_L^{\circ}/ kQ_L|=|P/kQ_L|
\]
where we have used Lemma \ref{QLP} in the last equality.

Also recall that the irreducible $L_{\wg}(k,0)$-modules are $\{L_{\wg}(k,\Lambda)\mid \Lambda \in P^k_+\}$.
Using Theorem \ref{quanequal} gives
\begin{equation*}
\begin{split}
{\rm glob} (L_{\wg}(k,0))\cdot |Q/kQ_L| &=\sum_{\Lambda\in P^k_+ }|Q/kQ_L| \cdot (\qdim_{L_{\wg}(k,0)} L_{\wg}(k,\Lambda) ) ^2\\
                                                              &=\sum_{\Lambda\in P^k_+ }|Q/kQ_L| \cdot (\qdim_{K(\g,k)}M^{\Lambda, \lambda} ) ^2\\
                                                              &=\sum_{\Lambda\in P^k_+ }\sum_{i\in Q/kQ_L}  (\qdim_{K(\g,k)}M^{\Lambda, \Lambda+\beta_i} ) ^2
\end{split}
\end{equation*}
where $\lambda$ is any fixed element in $\Lambda+Q.$
So we get
\begin{equation*}
\begin{split}
{\rm glob} (K(\g,k))&=\frac{{\rm glob} ( L_{\wg}(k,0))  |Q/kQ_L|^2}{{\rm glob} (V_{\sqrt{k}Q_L}) }\\
                                     &=\frac{{\rm glob} ( L_{\wg}(k,0))  |Q/kQ_L|^2}{|P/ Q||Q/kQ_L|}\\
                                      &=\frac{{\rm glob} ( L_{\wg}(k,0))  |Q/kQ_L|}{|P/ Q|}\\
                                     &=\frac{\sum_{\Lambda\in P^k_+ }\sum_{i\in Q/kQ_L}  (\qdim_{K(\g,k)}M^{\Lambda, \Lambda+\beta_i} ) ^2}{|P/Q|}.
\end{split}
\end{equation*}
It follows from Theorem \ref{DR}, Lemma \ref{iden3} that the identification in Theorem \ref{DR} is complete
and $K(\g,k)$ has exactly $\frac{|P^k_+||Q/kQ_L|}{|P/Q|}$ inequivalent irreducible  $K(\g,k)$-modules.
\end{proof}

Finally we determine the fusion rules among the irreducible modules for $K(\g,k).$
Let
$$L_{\wg}(k,\Lambda^1)\boxtimes L_{\wg}(k,\Lambda^2)=\sum_{\Lambda^3\in P^k_+ }N_{\Lambda^1,\Lambda^2}^{\Lambda^3}L_{\wg}(k,\Lambda_3)$$
where  $\Lambda^1,\Lambda^2\in P_+^k$
and  $N_{\Lambda^1,\Lambda^2}^{\Lambda^3}$ are the fusion rules for the irreducible $L_{\wg}(k,0)$-modules.

\begin{thm} Let $\Lambda^1,\Lambda^2\in P_+^k$ and $i,j\in Q/kQ_L.$ Then
$$M^{\Lambda^1,\Lambda^1+\beta_i}\boxtimes M^{\Lambda^2,\Lambda^2+\beta_j}
=\sum_{\Lambda^3\in P_+^k}N_{\Lambda^1,\Lambda^2}^{\Lambda^3}M^{\Lambda^3, \Lambda^1+\Lambda^2+\beta_i+\beta_j}.$$
Moreover, $M^{\Lambda^3, \Lambda^1+\Lambda^2+\beta_i+\beta_j}$ with $N_{\Lambda^1,\Lambda^2}^{\Lambda^3}\ne 0$
are inequivalent $K(\g,k)$-modules.
\end{thm}
\begin{proof} We first prove that $M^{\Lambda^3, \Lambda^1+\Lambda^2+\beta_i+\beta_j}$ with $N_{\Lambda^1,\Lambda^2}^{\Lambda^3}\ne 0$
are inequivalent $K(\g,k)$-modules for $\Lambda\in P_+^k.$  Note from Theorem \ref{DR}, Lemma \ref{iden3} and Theorem \ref{idenm} that
$$M^{\Lambda^3, \Lambda^1+\Lambda^2+\beta_i+\beta_j}=M^{\bar\Lambda, \Lambda^1+\Lambda^2+\beta_i+\beta_j}$$
for some $\Lambda^3,\bar\Lambda\in P_+^k$ if and only if $\bar\Lambda=(\Lambda^3)^{(s)}$ for some $s$ with $a_s=1$
and
$$\Lambda^1+\Lambda^2+\beta_i+\beta_j+k\Lambda_s-(\Lambda^1+\Lambda^2+\beta_i+\beta_j)\in kQ_L.$$
That is, $\Lambda_s\in Q_L.$ But this is impossible \cite{DR}.

For  any $\Lambda^3\in P_+^k,$ let $N_{M^{\Lambda^1,\Lambda^1+\beta_i}, M^{\Lambda^2,\Lambda^2+\beta_j} }^{M^{\Lambda^3,\Lambda^1+\Lambda^2+\beta_i+\beta_j}}$ be the fusion rules determined by the irreducible
$K(\g,k)$-modules $M^{\Lambda^1,\Lambda^1+\beta_i},$ $M^{\Lambda^2,\Lambda^2+\beta_j} $ and
$M^{\Lambda^3,\Lambda^1+\Lambda^2+\beta_i+\beta_j}.$ We claim that
$$N_{M^{\Lambda^1,\Lambda^1+\beta_i}, M^{\Lambda^2,\Lambda^2+\beta_j} }^{M^{\Lambda^3,\Lambda^1+\Lambda^2+\beta_i+\beta_j}}\geq N_{\Lambda^1,\Lambda^2}^{\Lambda^3}.$$
For short we set $\lambda=\Lambda^1+\beta_i$ and $\mu=\Lambda^2+\beta_j.$
By Proposition \ref{adl1} we know that
$$N_{\Lambda^1,\Lambda^2}^{\Lambda^3}\leq N_{V_{\sqrt{k}Q_L+\frac{1}{\sqrt{k}}\lambda}\otimes M^{\Lambda^1,\lambda}, V_{\sqrt{k}Q_L+\frac{1}{\sqrt{k}}\mu}\otimes M^{\Lambda^2,\mu}}^{L_{\wg}(k,\Lambda^3)}
 =N_{V_{\sqrt{k}Q_L+\frac{1}{\sqrt{k}}\lambda}\otimes M^{\Lambda^1,\lambda}, V_{\sqrt{k}Q_L+\frac{1}{\sqrt{k}}\mu}\otimes M^{\Lambda^2,\mu}}^{V_{\sqrt{k}Q_L+\frac{1}{\sqrt{k}}(\lambda+\mu)}\otimes M^{\Lambda^3,\lambda+\mu}}.$$
Using Proposition \ref{adl2} and the identity
$$N_{V_{\sqrt{k}Q_L+\frac{1}{\sqrt{k}}\lambda}, V_{\sqrt{k}Q_L+\frac{1}{\sqrt{k}}\mu}}^{V_{\sqrt{k}Q_L+\frac{1}{\sqrt{k}}(\lambda+\mu)}}=1$$
from \cite{DL} proves the claim.

From the discussion above, we see that $\sum_{\Lambda^3\in P_+^k}N_{\Lambda^1,\Lambda^2}^{\Lambda^3}M^{\Lambda^3, \Lambda^1+\Lambda^2+\beta_i+\beta_j}$
is a $K(\g,k)$-submodule of $M^{\Lambda^1,\Lambda^1+\beta_i}\boxtimes M^{\Lambda^2,\Lambda^2+\beta_j}.$
On the other hand, by Lemma \ref{quanequal} we have
\begin{eqnarray*}
&&\ \ \ \qdim_{K(g,k)} M^{\Lambda^1,\Lambda^1+\beta_i}\boxtimes M^{\Lambda^2,\Lambda^2+\beta_j} \\
&&=\qdim_{K(g,k)} M^{\Lambda^1,\Lambda^1+\beta_i}\qdim_{K(g,k)} M^{\Lambda^2,\Lambda^2+\beta_j} \\
&&=\qdim _{L_{\wg}(k,0)} L_{\wg}(k,\Lambda^1)\qdim _{L_{\wg}(k,0)} L_{\wg}(k,\Lambda^2)\\
&&=\qdim _{L_{\wg}(k,0)} L_{\wg}(k,\Lambda^1)\boxtimes  L_{\wg}(k,\Lambda^2)\\
&&=\sum_{\Lambda^3\in P^k_+ }N_{\Lambda^1,\Lambda^2}^{\Lambda^3}\qdim_{L_{\wg}(k,0)}L_{\wg}(k,\Lambda^3)\\
&&=\sum_{\Lambda^3\in P_+^k}N_{\Lambda^1,\Lambda^2}^{\Lambda^3}\qdim_{K(\g,k)}M^{\Lambda^3, \Lambda^1+\Lambda^2+\beta_i+\beta_j}.
\end{eqnarray*}
So the quantum dimension of the  submodule $\sum_{\Lambda^3\in P_+^k}N_{\Lambda^1,\Lambda^2}^{\Lambda^3}M^{\Lambda^3, \Lambda^1+\Lambda^2+\beta_i+\beta_j}$ of $M^{\Lambda^1,\Lambda^1+\beta_i}\boxtimes M^{\Lambda^2,\Lambda^2+\beta_j}$ equals to the quantum
dimension of  $M^{\Lambda^1,\Lambda^1+\beta_i}\boxtimes M^{\Lambda^2,\Lambda^2+\beta_j}.$
The Theorem follows immediately.
\end{proof}


\begin{thebibliography}{ABCDE}
\bibitem{ABD}T. Abe, G. Buhl and C. Dong, Rationality, regularity and $C_2$-cofiniteness, {\em Trans. AMS.} {\bf 356} (2004), 3391-3402.
\bibitem{ADL}T. Abe. C. Dong and H. Li, Fusion rules for the vertex operator
algebras $M(1)$ and $V_L^+,$ {\em Comm. Math. Phys.} {\bf 253} (2005),
171-219.
\bibitem{ALY1} T. Arakawa, C.H. Lam and H. Yamada, Zhu's algebra, $C_2$-algebra and $C_2$-cofiniteness of parafermion vertex operator algebras,  {\em Adv. Math.} {\bf 264} (2014), 261-295.
\bibitem{ALY2}T. Arakawa, C.H. Lam and H. Yamada, A characterazation of parafermion vertex operator algebras, preprint.
\bibitem{BK} B. Bakalov and A. Kirillov, Lectures on tensor categories and modular functors. {\em University Lecture Series} Vol 21, 2000.
\bibitem{B}
R. E. Borcherds, Vertex algebras, Kac-Moody algebras, and the Monster,
{\it Proc. Natl. Acad. Sci. USA} {\bf 83} (1986), 3068-3071.

\bibitem{CM} S. Carnahan and M. Miyamoto, Regularity of fixed-point vertex operator subalgebras, arXiv:1603.05645.


\bibitem{CL} H.Y. Chen and  C.H. Lam, Quantum dimensions and fusion rules of the VOA $ V^t_{L_{C \times D}},$
 arXiv:1409.0332.

\bibitem{C}R. Coquereaux, Global dimensions for Lie groups at level k and their conformally
exceptional quantum subgroups, {\em Revista de la Union Matematica Argentina}
{\bf 51} (2010), 17-42.
\bibitem{DMNO}A. Davydov, M. M\"{u}ger, D. Nikshych and V. Ostrik,  The Witt group of non-degenerate braided fusion categories, {\em J. Reine Angew. Math.} {\bf 677} (2013), 135-177.
\bibitem{D}
C. Dong, Vertex algebras associated with even lattices, \emph{J. Algebra}
\textbf{161} (1993), 245--265.

\bibitem{DJX1} C. Dong, X. Jiao and F. Xu, Quantum dimensions and quantum Galois theory, {\em Trans. AMS.} {\bf 365} (2013), 6441-6469.

\bibitem{DLY} C. Dong, C.H. Lam and H. Yamada, W-algebras related to
parafermion algebras, \emph{J. Algebra} {\bf 322} (2009), 2366-2403.

\bibitem{DLWY} C. Dong, C.H. Lam, Q. Wang and H. Yamada, The
structure of parafermion vertex operator algebras,  {\em J. Algebra} {\bf 323} (2010), 371-381.

\bibitem{DL}
C. Dong and J. Lepowsky, \emph{Generalized Vertex Algebras and Relative
Vertex Operators}, Progress in Math., Vol. 112, Birkh\"{a}user, Boston,
1993.
\bibitem{DLM0}C. Dong, H. Li and G. Mason, Compact automorphism groups of vertex operator algebras, {\em
International Math. Research Notices} {\bf 18} (1996), 913-921.
\bibitem{DLM1} C. Dong, H. Li and G. Mason,
Regularity of rational vertex operator algebras, {\em  Adv. Math.} {\bf 132} (1997), 148-166.

\bibitem{DLM2} C. Dong, H. Li and G. Mason,
Twisted representations of vertex operator algebras, {\em Math. Ann.}
{\bf  310} (1998), 571--600.

\bibitem{DLM3}C. Dong, H. Li and G. Mason,
Modular invariance of trace functions in orbifold theory and generalized
moonshine, {\em Comm. Math. Phys.} {\bf 214} (2000), 1-56.
\bibitem{DLM4} C. Dong, H. Li and G. Mason,
Vertex operator algebras and associative algebras, {\em J. Algebra}
{\bf 206} (1998), 67-96.

\bibitem{DLM5} C. Dong, H. Li and G. Mason, Twisted representations of vertex operator algebras and associative algebras, {\em
International Math. Research Notices,} {\bf 8} (1998), 389-397.

\bibitem{DLMM} C. Dong, H. Li, G. Mason and P. Montague, The radical of a vertex operator algebra,
in: {\em Proc. of the Conference on the Monster and Lie algebras at
The Ohio State University, May 1996,}  ed. by J. Ferrar and K. Harada,
Walter de Gruyter, Berlin-New York, 1998, 17-25.

\bibitem{DLN}
C. Dong, X. Lin and S.-H. Ng, Congruence Property In Conformal Field Theory,  arXiv:1201.6644.

\bibitem{DM}
C. Dong and G. Mason, On quantum Galois theory,  \emph{Duke Math. J.}
\textbf{86} (1997), 305--321.

\bibitem{DR} C. Dong and L. Ren, Representations of the parafermion vertex operator algebras,  arXiv:1410.3728.
\bibitem{DRX} C. Dong, L. Ren and F. Xu, On orbifold theory, arXiv:1507.03306.

\bibitem{DW1} C. Dong and Q. Wang, The structure of parafermion vertex operator algebras: general case, {\em Comm. Math. Phys.}
 {\bf 299} (2010), 783-792.

\bibitem{DW2} C. Dong and Q. Wang, On $C_2$-cofiniteness of the parafermion vertex operator algebras,  {\em J. Algebra} {\bf 328} (2011), 420-431.
\bibitem{DW3} C. Dong and Q. Wang, Parafermion vertex operator algebras, {\em Frontiers of Mathematics in China} {\bf 6(4)} (2011), 567-579.

\bibitem{DW4} C. Dong and Q. Wang, Quantum dimensions and fusion rules for parafermion vertex operator algebras,
{\em Proc. AMS} {\bf 144} (2016), 1483-1492.

\bibitem{ENO} P. Etingof, D. Nikshych and V. Ostrik, On fusion categories, {\em Ann. of Math.} {\bf 162} (2005), 581-642.

\bibitem{FHL} I. Frenkel, Y. Huang and J. Lepowsky, On axiomatic
approaches to vertx operator algebras and modules, {\em Mem. AMS} {\bf  104}, 1993.
\bibitem{FLM} I. B. Frenkel, J. Lepowsky and A. Meurman,
\emph{Vertex Operator Algebras and the Monster}, Pure and Applied Math.,
Vol. 134, Academic Press, Boston, 1988.

\bibitem{FZ}
I. B. Frenkel and Y. Zhu, Vertex operator algebras associated to
representations of affine and Virasoro algebras, \emph{Duke Math. J.}
\textbf{66} (1992), 123--168.

\bibitem{H1} Y. Huang, A theory of tensor
products for module categories for a vertex operator algebra, IV, { J. Pure Appl. Alg.} {\bf 100} (1995), 173-216.

\bibitem{H2} Y. Huang, Vertex operator algebras and the Verlinde conjecture, {Comm. Contemp. Math.} {\bf 10} (2008), 103-154.
\bibitem{H3} Y. Huang, Rigidity and modularity of vertex operator algebras. {Comm. Contemp. Math}, {\bf 10} (2008), 871-911.

\bibitem{HKL} Y. Huang, A. Kirillov Jr. and J. Lepowsky, Braided tensor categories and extensions of vertex operator algebras, arXiv 1406.3420.
\bibitem{HL1} Y. Huang and J. Lepowsky, A theory of tensor
products for module categories for a vertex operator algebra, I, {Selecta. Math. (N. S)} {\bf 1} (1995), 699-756.
\bibitem{HL2} Y. Huang and J. Lepowsky, A theory of tensor
products for module categories for a vertex operator algebra, II, {Selecta. Math. (N. S)} {\bf 1} (1995), 756-786.
\bibitem{HL3} Y. Huang and J. Lepowsky, A theory of tensor
products for module categories for a vertex operator algebra, III. {J. Pure Appl. Alg.} {\bf 100} (1995),
141-171.

\bibitem{H}J. Humphreys,  Introduction to Lie algebras and representation theory,
Graduate Texts in Mathematics, 9. Springer-Verlag, New York-Berlin, 1978.

\bibitem{K}
V. G. Kac, \emph{Infinite-dimensional Lie Algebras}, 3rd ed., Cambridge
University Press, Cambridge, 1990.
\bibitem{KO} A. Kirillov Jr. and V. Ostrik,  On a q-analogue of the McKay correspondence and the ADE classification of $sl_2$ conformal field theories, {\em Adv. Math.} {\bf 171} (2002), 183�227.

\bibitem{KM} M. Krauel and M. Miyamoto, A modular invariance property of multivariable trace functions for regular vertex operator algebras, {\em J. Alg.} {\bf 444} (2015), 124-142.

\bibitem{LL} J. Lepowsky and H. Li, \emph{Introduction to Vertex Operator Algebras
and Their Representations}, Progress in Math., Vol. 227, Birkh\"auser,
Boston, 2004.

\bibitem{LP} J. Lepowsky and M. Primc, Structure of the standard
modules for the affine Lie algebra $A_1^{(1)},$ \emph{Contemporary Math.}
\textbf{46}, 1985.

\bibitem{LW1} J. Lepowsky and R. L. Wilson, A new family
of algebras underlying the Rogers-Ramanujan identities and generalizations,
\emph{Proc. Natl. Acad. Sci. USA} \textbf{78} (1981), 7245-7248.

\bibitem{LW2} J. Lepowsky and R. L. Wilson, The structure of
standard modules, I: Universal algebras and the Rogers-Ramanujan
identities, \emph{Invent. Math.} \textbf{77} (1984), 199-290.

\bibitem{L2} H. Li, Local systems of twisted vertex operators, vertex operator superalgebras and twisted modules,   {\em Contemp. Math. AMS}  {\bf 193} (1996), 203-236.
 \bibitem{L1}H. Li, The physics superselection principle in vertex operator algebra theory, {\em  J. Alg.}
  {\bf 196} (1997), 436-457.

\bibitem{L3} H. Li, Certain extensions of vertex operator algebras of affine type, {\em Comm. Math. Phys.}
 {\bf 217} (2001), 653-696.

\bibitem{Lin} X. Lin, Mirror extensions of vertex operator algebras via tensor categories, arXiv:1411.7081.

\bibitem{M} M. Miyamoto, $ C_2$-cofiniteness of cyclic-orbifold models, {\em  Comm. Math. Phys.} {\bf  335}  (2015), 1279–1286. 

\bibitem{O} V. Ostrik, Module categories, weak Hopf algebras and modular invariants, {\em Transform.
Groups} {\bf 8} (2003), 177-206.

\bibitem{ZF} A. B. Zamolodchikov and V. A. Fateev,
Nonlocal (parafermion) currents in two-dimensional conformal quantum field
theory and self-dual critical points in $Z_N$-symmetric statistical
systems, \emph{Sov. Phys. JETP} \textbf{62} (1985), 215-225.

\bibitem{Z} Y. Zhu, Modular invariance of characters of vertex operator algebras,
{\em J. Amer, Math. Soc.} {\bf 9} (1996), 237-302.

\end{thebibliography}
\end{document}